\documentclass[11pt]{article}
\usepackage{amsmath}
\usepackage{latexsym}
\usepackage{epsfig}
\usepackage{amssymb}
\begin{document}
\makeatletter
\def\@begintheorem#1#2{\trivlist \item[\hskip \labelsep{\bf #2\ #1.}] \it}
\def\@opargbegintheorem#1#2#3{\trivlist \item[\hskip \labelsep{\bf #2\ #1}\ {\rm (#3).}]\it}
\makeatother
\newtheorem{thm}{Theorem}[section]
\newtheorem{alg}[thm]{Algorithm}
\newtheorem{conj}[thm]{Conjecture}
\newtheorem{lemma}[thm]{Lemma}
\newtheorem{defn}[thm]{Definition}
\newtheorem{cor}[thm]{Corollary}
\newtheorem{exam}[thm]{Example}
\newtheorem{prop}[thm]{Proposition}
\newenvironment{proof}{{\sc Proof}.}{\rule{3mm}{3mm}}

\title{On The 2-Spanning Cyclability Of Honeycomb Toroidal Graphs}
\author{Brian Alspach and Aditya Joshi\\School of Information and Physical Sciences\\
University of Newcastle\\Callaghan, NSW 2308\\
Australia\\brian.alspach@newcastle.edu.au\\aditya.joshi@uon.newcastle.edu}

\maketitle

\begin{abstract} A graph $X$ is 2-spanning cyclable if for any pair of distinct vertices $u$ and $v$ there
is a 2-factor of $X$ consisting of two cycles such that $u$ and $v$ belong to distinct
cycles.  In this paper we examine the 2-spanning cyclability of honeycomb toroidal
graphs.
\end{abstract}

\section{Introduction}

Graphs consist of vertices and edges joining pairs of distinct vertices such that neither loops
nor multiple edges are allowed.  If $X$ is a graph, its vertex set is denoted $V(X)$ and
its edge set is denoted $E(X)$.  The {\it order} of a graph $X$ is $|V(X)|$ and the {\it size}
is $|E(X)|$.  The number of edges incident with a given vertex $v$ is its {\it valency} and is 
denoted $\mathrm{val}(v)$.

An edge joining vertices $u$ and $v$ is denoted $[u,v]$.  Continuing in this vein,
a path of length $t$ from $u_0$ to $u_t$ is a connected subgraph of order $t+1$ all of whose 
vertices have valency 2 other than $u_0$ and $u_t$ which have valency 1.  The path is
denoted $[u_0,u_1,\ldots,u_t]$, where $[u_i,u_{i+1}]$ is an edge for $0\leq i\leq t-1$.  If we
start with a path of length $t$ from $u_0$ to $u_t$ and add the edge $[u_0,u_t]$, we obtain
a cycle of length $t+1$ and use the notation $[u_0,u_1,\ldots,u_t,u_0]$. 

A 2-factor in a graph $X$ is a spanning subgraph of $X$ such that every vertex has valency 2.
A 2-factor $F$ of $X$ {\it separates} a set $A$ of $k$ vertices in $V(X)$ if $F$ is composed of
$k$ cycles and $A$ intersects the vertex set of each cycle in $F$ in a single vertex.
  
\begin{defn}\label{kc} {\em A graph $X$ is }$k$-spanning cyclable {\em if for every $A\subseteq
V(X)$ such that $|A|=k$ there is a 2-factor of $X$ separating $A$.}
\end{defn}
  
The preceding concept has been studied because of the general problem of embedding cycles in
graphs.  Lin, Tan, Hsu and Kung \cite{L1} considered $k$-spanning cyclability for $n$-cubes.
Yang, Hsu, Hung and Cheng \cite{Y1} considered 2-spanning cyclability for generalized Petersen
graphs. In a recent paper, Qiao, Sabir and Meng studied $k$-cyclability of Cayley graphs on
symmetric groups whose connection sets are transposition trees.

Of course, a graph is Hamiltonian if and only if it is 1-spanning cyclable.  It is easy to see that
a regular graph of valency 3 cannot be 3-spanning cyclable because a cycle passing through a
vertex $v$ must contain two neighbors of $v$.  Hence, trivalent graphs are at most 2-spanning
cyclable.  Generalized Petersen graphs are a well-studied family of trivalent graphs and it is
natural that they were investigated in \cite{Y1}.  Another interesting family of trivalent graphs is
the honeycomb toroidal graphs (see \cite{A1} for a survey of this family).  A distinct advantage
they have over generalized Petersen graphs is that they are vertex-transitive, in fact, they are
Cayley graphs.  Also, they have been studied extensively by both computer scientists and mathematicians. 

\begin{defn}\label{hc} {\em The} honeycomb toroidal graph {\em $\mathrm{HTG}(m,n,\ell)$,
$m-\ell$ even, $n$ even, and $n\geq 4$, has
vertex set $\{u_{i,j}:0\leq i\leq m-1, 0\leq j\leq n-1\}$ and the following edge set $E$, where all
computations in the first subscript are done modulo $m$ and in the second subscript modulo $n$.
We have $[u_{i,j},u_{i,j+1}]\in E$ for all $i$ and $j$, $[u_{i,j},u_{i+1,j}]\in E$ for all $0\leq i\leq
m-2$ whenever $i+j$ is odd, and $[u_{m-1,j},u_{0,j+\ell}]$ whenever $m,j\mbox{ and }\ell$ have
the same parity.}
\end{defn}

Edges of the form $[u_{i,j},u_{i,j+1}]$ are called {\it vertical edges}, edges of the form $[u_{i,j},u_{i+1,j}]$
are called {\it flat edges}, and edges of the form $[u_{m-1,j},u_{0,j+\ell}]$ are called {\it jump
edges}.  It is well known \cite{M1} that honeycomb toroidal graphs are Hamiltonian, that is, they are
1-spanning cyclable.  As mentioned above, they cannot be 3-spanning cyclable.  Thus, the pertinent
question is:  Can we characterize the 2-spanning cyclable honeycomb toroidal graphs?

It is easy to see that $\mathrm{HTG}(m,n,\ell)$ is isomorphic to $\mathrm{HTG}(m,n,n-\ell)$.  Thus,
for subsequent results which are given in terms of $\ell$, a stronger result may arise by using $n-\ell$
instead of $\ell$ for $\ell>n/2$.

The parity of $m$ has played a fundamental role in several problems which have been investigated for 
honeycomb toroidal graphs.  This appears to be the case for 2-spanning cyclability as well.  The reason
for this is that it is easy to add two columns to go from $\mathrm{HTG}(m,n,\ell)$ to 
$\mathrm{HTG}(m+2,n,\ell)$ and preserve various properties.  Here is a brief description of the
process (see \cite{A1} for more details).  

Consider two consecutive columns of $\mathrm{HTG}(m,n,\ell)$ subscripted by $i\mbox{ and }i+1$,
and the flat edges $[u_{i,t_1},u_{i+1,t_1}], [u_{i,t_2},u_{i+1,t_2}],\ldots,[u_{i,t_k},u_{i+1,t_k}]$,
where $0\leq t_1<t_2< \cdots<t_k\leq n-1$.  Subdivide each of these flat edges twice so that the
edge $[u_{i,t_j},u_{i+1,t_j}]$ becomes the path $[u_{i,t_j},u_{a,t_j},u_{b,t_j},u_{i+1,t_j}]$.
Now add vertices to form columns subscripted by $a$ and $b$ with the obvious subscripts.

Remove each edge $[u_{a,t_j},u_{b,t_j}]$ from the corresponding 3-path and replace it with a path
from $u_{a,t_j}$ to $u_{b,t_j}$ by adding the vertical path from $u_{a,t_j}$ down to
$u_{a,1+t_{j-1}}$ followed by the edge $[u_{a,1+t_{j-1}},u_{b,1+t_{j-1}}]$ and then back
up column $b$ to $u_{b,t_j}$.  We then obtain paths from $u_{i,t_j}$ to $u_{i+1,t_j}$ that use
all the vertices of columns $a$ and $b$.  This operation is called the {\it vertical downward fill} for
columns $i$ and $i+1$.  The {\it vertical upward fill} is defined in an obvious analogous manner. 

The preceding does not work for $m=1$ as there are no flat edges in this case.  What we do to go
from $\mathrm{HTG}(1,n,\ell)$ to $\mathrm{HTG}(3,n,\ell)$ is the following.  We add two new
columns subscripted by 1 and 2.  When $m=1$, $\ell$ is odd and a jump edge has the form
$[u_{0,j},u_{0,j+\ell}]$, where $j$ is odd.  Replace that jump edge in $\mathrm{HTG}(1,n,\ell)$
with the jump edge $[u_{2,j},u_{0,j+\ell}]$ in $\mathrm{HTG}(3,n,\ell)$ and add the flat edge
$[u_{0,j},u_{1,j}]$.  It is obvious how to carry out vertical fills to obtain an $\mathrm{HTG}(3,n,\ell)$.
See Figure 1 below for an example.  

\section{One Column}

As noted above, we add two columns at a time which makes $m=1$ and $m=2$ base cases for
subsequent results.  We consider $m=1$ in this section.  Of course, $\ell$ is odd when $m=1$.
We ignore $\ell=1$ as $\mathrm{HTG}(1,n,1)$ is a multigraph.  We completely settle the cases
for $\ell=3$ and $\ell=5$.  For $\ell\geq 7$ we establish an asymptotic result.  An important convention
we adopt for this section is that we denote $u_{0,j}$ by $v_j$ for all $0\leq j\leq n-1$ in order to
simplify the notation.

We first give two easy negative results which are used later and illustrate basic techniques.

\begin{prop}\label{half} When $n\equiv 2(\mbox{\em mod }4)$, $\mathrm{HTG}(1,n,n/2)$
is not 2-spanning cyclable.
\end{prop}
\begin{proof} The graph $\mathrm{HTG}(1,6,3)$ is not 2-spanning cyclable because it is bipartite
and there is no 2-factor consisting of two cycles.  We assume that $n\geq 10$.  We claim there is
no 2-factor separating $v_0$ and $v_{n/2}$.  If there were, then there would be a cycle $C_1$ containing
$v_0$ and a distinct cycle $C_2$ containing $v_{n/2}$.  Thus, the edge $[v_0,v_{n/2}]$ is not used
implying that the 2-path $[v_{n-1},v_0,v_1]$ is contained in $C_1$ and the path $[v_{n/2-1},v_{n/2},v_{n/2+1}]$
lies in $C_2$.  

This, of course, forces both diameters $[v_1,v_{n/2+1}]$ and $[v_{n-1},v_{n/2-1}]$ to not be used.
It is now easy to see that this continues and no diameter appears in either $C_1$ or $C_2$ which
contradicts that $C_1$ and $C_2$ are distinct.     \end{proof}

\begin{prop}\label{ahalf} When $n\equiv 0(\mbox{\em mod }4)$ and $n\neq 8$, $\mathrm{HTG}(1,n,(n-2)/2)$
is not 2-spanning cyclable.
\end{prop}
\begin{proof} When $n=4$, $\mathrm{HTG}(1,4,1)$ is a multigraph and also is not 2-spanning cyclable.
It is easy to verify that $\mathrm{HTG}(1,8,3)$ is 2-spanning cyclable so we assume $n\geq 12.$  We now
show there is no 2-factor separating $v_1$ and $v_2$.

Because the edge $[v_1,v_2]$ cannot be used, the cycle $C_1$ containg $v_1$ must contain the 2-path
$[v_0,v_1,v_{n/2}]$ and the cycle $C_2$ containing $v_2$ must contain the 2-path $[v_{(n+6)/2},
v_2,v_3]$.  Two of the edges incident with $v_{(n+4)/2}$ also are incident with vertices of $C_2$ which
implies $v_{(n+4)/2}$ is on $C_2$.  Similarly, $v_{(n+2)/2}$ is on $C_1$.  Hence, the edge
$[v_{(n+2)/2},v_{(n+4)/2}]$ cannot be used.  This forces both $C_1$ and $C_2$ to be 4-cycles.  This
is a contradiction as $n>8$.  The conclusion follows.     \end{proof}

\medskip

Some words about verifying 2-spanning cyclability are in order.  The graph $\mathrm{HTG}(m,n,\ell)$
is vertex-transitive from which it follows that it suffices to show there is a 2-factor separating $u_{0,0}$ and an arbitrary
distinct vertex of the graph.   The permutation $\rho$ which rotates the vertices two positions, that is, 
$\rho(u_{i,j})=u_{i,j+2}$ clearly is an automorphism of $\mathrm{HTG}(m,n,\ell)$.  Similarly, the
permutation $\pi$ defined by $\pi(u_{i,j})=u_{i+1,j+1}$ for $i\neq m-1$, and  $\pi(u_{m-1,j})=u_{0,j+1+\ell}$
is an automorphism.

The group $G=\langle\rho,\pi\rangle$ generated by $\rho$ and $\pi$ is abelian and has two orbits.  One orbit 
contains the vertices whose subscripts have even sum and the other orbit contains the vertices whose subscripts 
have odd sum.  The action of $G$ restricted to each orbit is regular, that is, for two vertices $u_{i,j}$ and
$u_{i',j'}$ in the same orbit, there is a unique element of $G$ mapping $u_{i,j}$ to $u_{i',j'}$.  Finally,
the vertical cycles form a block system for $G$ and the image of any vertical cycle maintains the cyclic
order of the vertices under the action of any element of $G$.  This is very useful.  For example, if we have a 2-factor
separating $u_{i,j}$ and $u_{i,j+1}$, where $i+j$ is even, then there is a 2-factor separating $u_{0,0}$
and $u_{0,1}$.  It is obtained by letting $g\in G$ act on the 2-factor where $g(u_{i,j})=u_{0,0}$.

\begin{lemma}\label{O3} The honeycomb toroidal graph $\mathrm{HTG}(1,n,3)$ is
2-spanning cyclable if and only if $n\equiv 0(\mbox{mod }4)$ and $n>6$.
\end{lemma}
\begin{proof} When $n=4$ or $n=6$, $\mathrm{HTG}(1,n,3)$ does not even have a 2-factor
composed of two cycles because it is bipartite.  So we restrict ourselves to $n\geq 8$.  Let $C_1$
be the 4-cycle $[v_0,v_{n-3},v_{n-2},v_{n-1},v_{0}]$.  Start a second cycle with the
4-cycle $[v_1,v_2,v_3,v_4,v_1]$.  If $n=8$, we stop.  If $n>8$,
replace the edge $[v_3,v_4]$ with the path $[v_3,v_6,v_5,v_4]$
producing a cycle of length 6.  It is clear that we may continue in this manner until we reach a
2-factor composed of $C_1$ and a second cycle $C_2$ of length $n-4$.  Note that $C_1$ and
$C_2$ separate $v_0$ and $v_j$ for $j=1,2,3,\ldots,n-4$. 

As mentioned earlier, we need only consider $v_0$ and the three
vertices $v_{n-1},v_{n-2}\mbox{ and }v_{n-3}$.  If we apply $\rho^{-1}$ to the preceding 2-factor, 
we obtain a 2-factor that separates $v_0$ and both $v_{n-2}$ and $v_{n-3}$.  

Thus, the crucial pair with regard to separation is $v_0$ and $v_{n-1}$.  We consider the
problem of creating cycles $C_3$ and $C_4$ such that $v_0\in V(C_3)$, $v_{n-1}\in V(C_4)$,
and $C_3$ and $C_4$ form a 2-factor.

$C_3$ must contain the subpath $[v_{n-3},v_0,v_1]$ and $C_4$ must contain the subpath
$[v_{n-2},v_{n-1},v_2]$ because the edge $[v_{n-1},v_0]$ cannot be used.  Then
the edge $[v_1,v_2]$ cannot be used which forces us to add the edge $[v_1,v_4]$
to $C_3$ and the edge $[v_2,v_3]$ to $C_4$.  The edges continue to be forced and it is
easy to see that $C_3$ must continue to use $v_4, v_8, v_{12}$ and so on.  The only
way $C_3$ and $C_4$ do  not collide is if and only if $n\equiv 0(\mbox{mod }4)$.    \end{proof}

\medskip

Lemma \ref{O3} determines the exact values of $n$ for which $\mathrm{HTG}(1,n,3)$ is 2-spanning
cyclable.   When $\ell=5$, it is not difficult to show that $\mathrm{HTG}(1,n,5)$ fails to be
2-spanning cyclable only for $n\in\{10,12,14,20\}$ for $n\geq 10$.  The cases of $n=10,12$ are
covered by Propositions \ref{half} and \ref{ahalf}, other values are covered by Theorem \ref{alpha}
below, and the remaining values may be done manually.

This indicates it is likely that for each odd $\ell>3$, there are a few sporadic small values of $n$ for
which $\mathrm{HTG}(1,n,\ell)$ fails to be 2-spanning cyclable.  If this is the case, then trying to
determine the exact parameters $(1,n,\ell)$ for which the corresponding honeycomb toroidal graph
is 2-spanning cyclable becomes a tedious task.  Consequently, we complete this section in search of a
universal asymptotic result.

\begin{lemma}\label{even} If $x$ and $y$ are positive even integers satisfying $\mathrm{gcd}(x,y)=2$, 
then every even $n>\frac{xy}{2}-x-y$ may be written as a linear combination of $x$ and $y$ with nonnegative
coefficients.
\end{lemma}
\begin{proof} Because $\mathrm{gcd}(x/2,y/2)=1$, from \cite{W1} we know that the largest integer not
expressible as a linear combination of $x/2$ and $y/2$ with nonnegative coefficients (the Frobenius number)
is $(xy-2x-2y)/4$.  Doubling this gives us the largest even integer not expressible as a linear combination of
$x$ and $y$ with nonnegative coefficients which completes the proof.    \end{proof}

\begin{thm}\label{alpha}  Let $\mathrm{HTG}(1,n,\ell)$ satisfy $\ell$ odd and $\ell>3$.  If $n=2\ell+\alpha$,
where $\alpha$ is a linear combination of $\ell+1$ and $3\ell+1$ with nonnegative coefficients and
$\alpha>0$, then $\mathrm{HTG}(1,n,\ell)$ is 2-spanning cyclable.
\end{thm}
\begin{proof} Let $\mathrm{HTG}(1,n,\ell)$ satisfy the hypotheses.  Define $C_1$ to be the 6-cycle
$[v_{\ell-3},v_{\ell-2},v_{\ell-1},v_{n-1},v_{n-2},v_{n-3},v_{\ell-3}]$.  We now show 
there is a single cycle $C_2$ containing the remaining $n-6$ vertices, thereby giving us a 2-factor.

Two paths which play a central role in the construction of $C_2$ are defined as follows. For $x$
even, the path $$P(x)=[v_x,v_{x-1},v_{x-2},\dots,v_{x-\ell+4},v_{x+4},v_{x+3},v_{x+2},v_{x+1},v_{x+1+\ell}]$$
and the path

\medskip 

$Q(x)=[v_x,v_{x-1},v_{x-2},\ldots,v_{x-\ell+4},v_{x+4},v_{x+5},\ldots,v_{x+\ell+1},
v_{x+1},v_{x+2},v_{x+3},\\
v_{x+3+\ell},v_{x+2+\ell},v_{x+2+2\ell},v_{x+3+2\ell},v_{x+4+2\ell},v_{x+4+\ell},v_{x+5+\ell},
\ldots,v_{x+1+2\ell},v_{x+1+3\ell}]$.

\medskip

\noindent Note that $P(x)$ has length $\ell+1$ and $Q(x)$ has length $3\ell+1$.

Start building $C_2$ with the path $[v_{n-4},v_{n-5},\ldots,v_{n-\ell},v_0,v_1,v_2,\ldots,v_{\ell-4}].$
Note that the unused vertices are $v_{\ell},v_{\ell+1},v_{\ell+2},\ldots,v_{n-\ell-1}$ which is a segment
of $\alpha$ consecutive vertices because $n-\ell=\ell+\alpha$.  So we need a path of length $\alpha+1$
from $v_{\ell-4}$ to $v_{n-4}$ passing through all the unused vertices.  We start by adding the edge 
$[v_{\ell-4},v_{2\ell-4}]$ to the current termination of $C_2$.  Now we need a path of length $\alpha$ 
from $v_{2\ell-4}$ to $v_{n-4}$ using the remaining unused vertices.  Fortunately $\alpha$ is a linear 
combination of $\ell+1$ and $3\ell+1$ with nonnegative coefficients which means we may use $P(x)$ and 
$Q(x)$ paths.  A careful look at these paths is appropriate.

The path $P(x)$ uses all the vertices with increasing subscript from $v_{x-\ell+4}$ to $v_{x+4}$ plus
the vertex $v_{x+1+\ell}$ leaving a gap between $v_{x+4}$ and $v_{x+1+\ell}$. The path $Q(x)$ uses
all the vertices with increasing subscript from $v_{x-\ell+4}$ to $v_{x+4+2\ell}$ plus the vertex 
$v_{x+1+3\ell}$. If we attach the path $P(x+1+\ell)$ or $Q(x+1+\ell)$ to $P(x)$, the gap created by 
$P(x)$ is filled in addition to the segment it covers.  Similarly, if we attach the path $Q(x+1+3\ell)$ to
$Q(x)$, the gap created by $Q(x)$ is filled in addition to the segment it covers.  

Let $\alpha=\beta(\ell+1)+\gamma(3\ell+1)$, where $\beta,\gamma,\geq 0$ and at least one of them is
positive.  If $\beta>0$, continue $C_2$ by adding the path $P(2\ell-4)$.  This path terminates at $v_{3\ell-3}$.
If $\beta>1$, then we attach $P(3\ell-3)$.  If $\beta=1$, we attach $Q(3\ell-3)$.  We then continue this
process until we have attached $\beta$ $P(x)$ paths and $\gamma$ $Q(x)$ paths.  If $\beta=0$, we
begin the process with $Q(2\ell-4)$.  

Because the process is continually filling in consecutive vertices other than the last edge at each iteration
and $\beta(\ell+1)+\gamma(3\ell+1)=\alpha$, after $\beta+\gamma$ iterations, the largest subscripted
vertex is $v_{n-\ell-1}$ and the terminal vertex before adding the very last edge is $v_{n-\ell-4}$.
Adding the last edge takes us to $v_{n-4}$ producing the cycle $C_2$.

The cycles $C_1$ and $C_2$ separate $v_0$ and the six vertices of $C_1$.  The 2-factor made up of
the two cycles $\rho(C_1)$ and $\rho(C_2)$ separates $v_0$ from all remaining vertices except 
$v_1,v_{\ell}\mbox{ and }v_{\ell+1}$.  The 2-factor formed by $\rho^2(C_1)$ and $\rho^2(C_2)$
separate $v_0$ and $v_1$.  Using the power of $\rho$ that takes $v_{\ell-3}$ to $v_0$ yields a
2-factor separating $v_0$ from both $v_{\ell}$ and $v_{\ell+1}$.  This completes the proof.
\end{proof}

\begin{cor}\label{asymp} If odd $\ell>3$, then for all $n>\frac{3\ell^2-4\ell-3}{2}$, 
$\mathrm{HTG}(1,n,\ell)$ is 2-spanning cyclable.
\end{cor}
\begin{proof} If $n>2\ell+\frac{(\ell+1)(3\ell+1)}{2}-(\ell+1)-(3\ell+1)$, then $n>2\ell+\alpha$, 
where $\alpha>\frac{(\ell+1)(3\ell+1)}{2}-(\ell+1)-(3\ell+1)$.  By Lemma \ref{even}, $\alpha$ 
may be written as a linear combination of $\ell+1$ and $3\ell+1$ with nonnegative coefficients.  
Theorem \ref{alpha} then implies that $\mathrm{HTG}(1,n,\ell)$ is 2-spanning cyclable.  Carrying
out the algebra yields $\frac{3\ell^2-4\ell-3}{2}$.    \end{proof}

\bigskip

The lower bound is not tight.  To see this recall the information about $\ell=5$.  The lower bound from
Corollary \ref{asymp} is 26, whereas $\mathrm{HTG}(1,n,5)$ is 2-spanning cyclable for all even 
$n\geq 22$.
 
\section{Two Columns}

The outline for this section is the same as that for the preceding section.  Exact results are obtained for
the two smallest values of $\ell$.  This is followed by a result which enables us to obtain an asymptotic
result for all other even values of $\ell$.  We begin with the following immediate result.

\begin{lemma}\label{zero}  The graph $\mathrm{HTG}(2,n,0)$ is 2-spanning cyclable for all even $n\geq 4$.
\end{lemma}

When $m=2$, the value $\ell=2$ leads to negative results as does $\ell=3$ for $m=1$, but the big difference
is that $\ell=2$ is really negative as shown by the following result.

\begin{lemma}\label{two} The graph $\mathrm{HTG}(2,n,2)$ is 2-spanning cyclable if and only if $n=4$.
\end{lemma}
\begin{proof} It is straightforward to verify that $\mathrm{HTG}(2,4,2)$ is 2-spanning cyclable.  Consider
$n\geq 6$.  Let $C_1$ and $C_2$ be two cycles of a 2-factor that separate $u_{0,0}$ and $u_{0,1}$, where $u_{0,0}$
belongs to $C_1$.  Then $C_1$ must contain the 2-path $[u_{0,n-1},u_{0,0},u_{1,n-2}]$ and $C_2$
must contain the 2-path $[u_{0,2},u_{0,1},u_{1,1}]$.  

The vertex $u_{1,n-1}$ must be in $C_1$ and the vertex $u_{1,0}$ must be in $C_2$.  Thus, the edge
$[u_{1,0},u_{1,n-1}]$ cannot be in the 2-factor.  This forces $C_1$ and $C_2$ to be 4-cycles which means
they cannot span $\mathrm{HTG}(2,n,2)$ as the order is at least 12.  This completes the proof.     \end{proof}

\medskip
 
It is not difficult to determine that when $\ell=4$, the only values of $n$ for which $\mathrm{HTG}(2,n,4)$
fails to be 2-spanning cyclable are $n=6$ and $n=10$.  All values of even $n$ for $n\geq 12$ are covered by the
next theorem.  We leave the details for the smaller values of $n$ to the reader with the hint that $u_{0,0}$ and 
$u_{0,9}$ cannot be separated in $\mathrm{HTG}(2,10,4)$.
The next result is the analogue of Theorem \ref{alpha} for $m=2$.  The proof is similar.

\begin{thm}\label{beta}  Let $\ell>2$ be even.  If $n=2\ell+\alpha$, where $\alpha$ is the sum of
nonnegative multiples of $\ell+2$ and $\ell$, then $\mathrm{HTG}(2,n,\ell)$ is 2-spanning cyclable.
\end{thm}
\begin{proof}  Let $\mathrm{HTG}(2,n,\ell)$ satisfy the hypotheses.  The two vertical cycles separate
the vertices of the respective columns.  It suffices to show that we can separate $u_{0,0}$ and any other
vertex of the form $u_{0,j}$, $j\neq 0$.  Let $$C_1=[u_{0,0},u_{0,n-1},u_{1,n-1},u_{1,0},u_{1,1},\ldots,
u_{1,\ell-3},u_{0,\ell-3},u_{0,\ell-4},\ldots,u_{0,0}].$$  

We claim there is a cycle $C_2$ through the  unused vertices.  Let $$P(x)=[u_{0,x},u_{0,x+1},\ldots,u_{0,x+\ell},
u_{1,x},u_{1,x+1},\ldots,u_{1,x+\ell+1},u_{0,x+\ell+1},u_{0,x+\ell+2}],$$ where $x$ is even.  Note
that $P(x)$ has length $2\ell+4$ and uses all the vertices of $\ell+2$ consecutive rows.  Let $R(x)$ be the
subpath of $P(x)$ obtained by deleting the last three edges so that it terminates at $u_{1,x+\ell}$ and uses all
the vertices on $\ell+1$ consecutive rows.  Finally,
let $$Q(x)=[u_{1,x},u_{0,x+\ell},u_{0,x+\ell-1},\ldots,u_{0,x+1},u_{1,x+1},u_{1,x+2},\ldots,u_{1,x+\ell}],$$
where $x$ is even.  This path has length $2\ell$ and uses all the vertices of $\ell$ consecutive rows.

The cycle $C_1$ has length $2\ell-2$ and uses all the edges of rows $0,1,\ldots,\ell-3$ and $n-1$.  This
leaves the vertices on rows $\ell-2$ through $n-2$ unused which consists of $\ell+1+\alpha$ consecutive
rows.  Let $\alpha=\beta(\ell+2)+\gamma\ell$, where $\beta,\gamma\geq 0$.

We now describe the cycle $C_2$.  If $\beta>0$, then start with $P(\ell-2)$ and juxtapose $\beta$ paths
of this type.  The path constructed so far terminates at $u_{0,\ell-2+\beta(\ell+2)}$ and uses the vertices
on $\beta(\ell+2)$ consecutive rows.  Now add the path $R(\ell-2+\beta(\ell+2))$ which uses the vertices
of the next $\ell+1$ consecutive rows.  Then juxtapose $\gamma$ paths of $Q$-type.  The path now
terminates at $u_{1,n-2}$ and adding the jump edge takes us back to $u_{0,\ell-2}$, thereby constructing
the cycle $C_2$.  If $\beta=0$, then start the path with $R(\ell-2)$ and complete it to a cycle in the obvious
way.  

We now examine separation properties.  In the special case that $\alpha=0$, the path $R(\ell-2)$ completes
to $C_2$ itself.  The 2-factor $F$ formed by $C_1\cup C_2$ separates $u_{0,0}$
and all vertices of the form $u_{0,j}$ for $j\in\{\ell-2,\ell-1,\ldots,n-2\}$.  Taking the power of $\rho$ which
maps $u_{0,\ell-4}$ to $u_{0,0}$, we see that $u_{0,0}$ is separated from $u_{0,j}$ for $j\in\{2,3,\ldots,
n-\ell+1\}$.  This leaves only $u_{0,0}$ and $u_{0,1}$.  They are separated because $u_{0,n-2}$ and
$u_{0,n-1}$ are separated by $F$.      \end{proof}

\begin{cor}\label{betamax} Let $\ell>2$ be even.  If $n>\frac{\ell^2+2\ell-4}{2}$, then $\mathrm{HTG}
(2,n,\ell)$ is 2-spanning cyclable.
\end{cor}
\begin{proof} Lemma \ref{even} implies that any positive even integer bigger than $\frac{\ell^2\ell}{2}
-\ell-\ell-2$ can be written as a linear combination of $\ell$ and $\ell+2$ with nonnegative coefficients because
$\mathrm{gcd}(\ell,\ell+2)=2$.  The result then follows from Theorem \ref{beta} by letting $\alpha$
be the preceding value and simplifying the expression.      \end{proof}
 
\section{Odd Number Of Columns}

This sections deals with an arbitrary odd number of columns.  We first look at cases for which the exact
answer is known and then look at asymptotic results.   The value $\ell=1$ was not considered earlier as
$\mathrm{HTG}(1,n,1)$ is not a graph.  We do have a graph for all odd $m>1$.

\begin{thm}\label{L1} The honeycomb toroidal graph $\mathrm{HTG}(m,n,1)$ is 2-spanning cyclable
if and only if $m\geq 3$ is odd and $n>4$.
\end{thm}
\begin{proof} We claim that $\mathrm{HTG}(m,4,1)$ is not 2-spanning cyclable for all odd $m\geq 3$.
It suffices to show there is no 2-factor separating $u_{0,0}$ and $u_{m-1,3}$.  If there was such a 2-factor,
then the edge $[u_{0,0},u_{m-1,3}]$ cannot be in the 2-factor.  This implies that the 2-path
$[u_{0,1},u_{0,0},u_{0,3}]$ must lie in one cycle and $[u_{m-1,0},u_{m-1,3},u_{m-1,2}]$ must lie in
the other cycle.  Then the jump edge $[u_{m-1,1},u_{0,2}]$ cannot belong to the 2-factor as it would be
forced to be in both cycles.  But that implies column 0 is one of the cycles and column $m-1$ is one of
the cycles.  This is a contradiction as $m>2$.

Assume that $n>4$ (recall that $n$ always is even).  We shall first provide a 2-factor for $m=3$ and use
vertical fills for all other values of $m$.  Let $C_1=[u_{0,0},u_{0,1},u_{1,1},u_{1,0},u_{1,n-1},
u_{0,n-1},u_{0,0}]$.  The second cycle $C_2$ is $$[u_{2,0},u_{2,1},u_{0,2},\ldots,u_{0,n-2},u_{2,n-3},
\ldots,u_{2,2},u_{1,2},\ldots,u_{1,n-2},u_{2,n-2},u_{2,n-1},u_{2,0}].$$

The 2-factor $C_1\cup C_2$ separates $u_{0,0}$ from all vertices except the other vertices of the 6-cycle 
$C_1$.  We leave it up to the reader to make the straightforward observation that we can separate $u_{0,0}$
from these other five vertices by considering the action of the group $G$ on the 2-factor.

Using vertical fills on the flat edges $[u_{1,2},u_{2,2}]$ and $[u_{1,n-2},u_{2,n-2}]$, we have all
the vertices of the two new columns belonging to the extended $C_2$.  The result follows because we can
pick up the remaining five vertices as we did when $m=3$.   \end{proof}

\medskip

We now consider extending the results for $m=1$ to any odd $m$.  It is worth looking at Figure 1 to form an idea of how to approach the problem of extending results for one column to
an odd number of columns.  First, note that we start with a 2-factor consisting of two cycles for $\mathrm{HTG}
(1,12,3)$.  We then use a vertical upward fill to generate a 2-factor of $\mathrm{HTG}(3,12,3)$ also consisting
of two cycles.  The main feature of the transition is that if the two cycles for $\mathrm{HTG}(1,12,3)$ separate
$u_{0,0}$ and $u_{0,j}$, then the two cycles for $\mathrm{HTG}(3,12,3)$ separate $u_{0,0}$ and
$u_{i,j}$ for $i=0,1,2$.  This is a feature we need to capture in general arguments.

It is now apparent that a special case has arisen.  Namely, separating $u_{0,0}$ and either $u_{1,0}$ or
$u_{2,0}$ will not be inherited from results about $\mathrm{HTG}(1,12,3)$.  So independent arguments
will be needed for the vertices at level zero.

\begin{thm}\label{L3} The honeycomb toroidal graph $\mathrm{HTG}(m,n,3)$ is 2-spanning cyclable
if and only if 
\begin{itemize}
\item $m=1$, $n\equiv 0(\mbox{mod }4)$ and $n>6$,
\item $m\geq 5$, $m$ odd and $n=6$, and
\item $m\geq 3$, $m$ odd and $n\geq 8$.
\end{itemize}
\end{thm} 

\begin{proof}  First note that $n=4$ is not included for the simple reason that $\mathrm{HTG}(m,4,3)$
is isomorphic to $\mathrm{HTG}(m,4,1)$ and Theorem \ref{L1} implies that none of them are
2-spanning cyclable.
Lemma \ref{O3} covers the first item and $\mathrm{HTG}(1,6,3)$ in the second item. 
It is straightforward, though tedious, to show there is no 2-factor separating $u_{0,0}$ and $u_{0,5}$ in 
$\mathrm{HTG}(3,6,3)$.  We leave the verification to the reader.

\begin{picture}(200,300)(-80,-50)
\multiput(0,0)(0,20){12}{\circle*{5}}
\put(0,20){\line(0,1){40}}
\multiput(0,20)(0,40){3}{\qbezier(0,0)(-15,30)(0,60)}
\put(0,80){\line(0,1){20}}
\put(0,120){\line(0,1){40}}
\put(0,180){\line(0,1){40}}
\qbezier(0,0)(20,90)(0,180)
\qbezier(0,0)(40,110)(0,220)
\multiput(100,0)(20,0){3}{\multiput(0,0)(0,20){12}{\circle*{5}}}
\multiput(100,20)(0,40){3}{\line(1,0){20}}
\put(100,180){\line(1,0){20}}
\multiput(100,80)(0,40){3}{\line(2,-3){40}}
\multiput(100,20)(0,100){2}{\line(0,1){40}}
\qbezier(100,0)(120,90)(140,180)
\put(100,180){\line(0,1){40}}
\qbezier(100,0)(75,110)(100,220)
\multiput(120,20)(0,40){2}{\line(0,1){20}}
\multiput(120,40)(0,40){2}{\line(1,0){20}}
\put(100,80){\line(0,1){20}}
\multiput(140,20)(0,40){2}{\line(0,1){20}}
\multiput(120,100)(20,0){2}{\line(0,1){60}}
\put(120,160){\line(1,0){20}}
\multiput(120,180)(20,0){2}{\line(0,1){50}}
\multiput(120,0)(20,0){2}{\line(0,-1){15}}
\put(120,0){\line(1,0){20}}

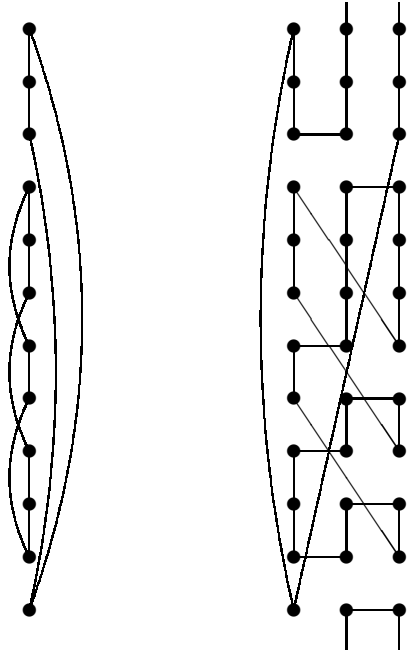
\put(-50,-35){{\sc Figure 1.} Transition from $\mathrm{HTG}(1,12,3)$ to $\mathrm{HTG}(3,12,3)$}
\end{picture}

So we continue with the second item considering $m\geq 5$.  Let $C_1$ be the cycle \begin{eqnarray*}
[u_{0,0},u_{0,1},u_{0,2},u_{4,5},u_{4,0},u_{4,1},u_{4,2},u_{3,2},u_{3,3},u_{2,3},\\u_{2,2},
u_{2,1},u_{3,1},u_{3,0},u_{3,5},u_{3,4},u_{4,4},u_{4,2},u_{0,0}]
\end{eqnarray*}
in $\mathrm{HTG}(5,6,3)$.  The complementary cycle is $$C_2=[u_{1,0},u_{1,1},u_{1,2},u_{1,3},u_{0,3},
u_{0,4},u_{0,5},u_{1,5},u_{1,4},u_{2,4},u_{2,5},u_{2,0},u_{1,0}].$$

The facts that every vertex of the 1-column lies in $C_2$ and the columns form blocks for the group $G$
imply that there are 2-factors separating $u_{0,0}$ and every vertex of the form $u_{i,j}$ for $i\neq 0$ and
$0\leq j\leq 5$.   Also, $u_{0,0}$ is separated from $u_{0,3},u_{0,4}\mbox{ and }u_{0,5}$ in the 2-factor
formed by $C_1\cup C_2$.  Finally, $\rho^{-1}$ acting on the latter 2-factor yields a 2-factor separating
$u_{0,0}$ from $u_{0,1}\mbox{ and }u_{0,2}$ completing the proof for $\mathrm{HTG}(5,6,3)$.

Starting with $C_1$ and $C_2$, we may use repeated downward or upward fills between the last two
columns to obtain 2-factors for any odd value of $m>5$ that easily lead to the conclusion that $\mathrm{HTG}(m,6,3)$
is 2-spanning cyclable.  This completes the proof of the second item.

To prove the last item, we start with $m=3$.  Let $C_3$ be the 12-cycle \[[u_{0,0},u_{0,1},u_{0,2},
u_{0,3},u_{1,3},u_{1,2},u_{1,1},u_{1,0},u_{2,0},u_{2,n-1},u_{2,n-2},u_{2,n-3},u_{0,0}].\]
The next step is to show that for all even $n\geq 8$ there is a complementary cycle yielding a 2-factor.

Let $P(x)$, $x$ odd, denote the path $$[u_{2,x},u_{0,x+3},u_{0,x+4},u_{1,x+4},u_{1,x+3},u_{2,x+3},
u_{2,x+4}].$$  Now $n=8$ is a special case.  Obtain the complementary cycle by starting with $P(1)$ from
which the last edge $[u_{2,4},u_{2,5}]$ is deleted and is replaced by the path from $u_{2,4}$ down to
$u_{2,1}$.  Then use vertical extension on the edge $[u_{0,5},u_{1,5}]$ to reach $[u_{0,7},u_{1,7}]$
and we now have a complementary 12-cycle.

For $n\geq 10$, use $P(x)$ for odd $x$ from $x=1$ through $x=n-9$.  At this point those paths form
a single path whose terminal vertices are $u_{2,n-5}$ and $u_{2,n-7}$ and use the remaining vertices on
levels 1 through level $n-5$.  Now add the edge $[u_{2,n-7},u_{0,n-4}]$ and the 2-path $[u_{2,n-5},
u_{2,n-4},u_{1,n-4}]$.  Complete the path to a cycle by performing the obvious vertical extension to
level $n-1$.  Call this cycle $C_4$.  (See Figure 2 for the case $n=12$.)

Clearly $F=C_3\cup C_4$ separates $u_{0,0}$ from every vertex on $C_4$.  If we let $\rho^{-1}$ act on
the 2-factor $F$, then $u_{0,0}$ is separated from $u_{0,2},u_{0,3},u_{1,2},u_{1,3},u_{2,0}$ and
$u_{2,n-1}$ in $\rho^{-1}(F)$.  The 2-factor $(\rho\pi)^{-1}(F)$ separates $u_{0,0}$ from
$u_{0,1},u_{1,0},u_{1,1}$ and $u_{2,n-2}$.  Finally, $\pi^{-2}(F)$ separates $u_{0,0}$ and
$u_{2,n-3}$.  Thus, $\mathrm{HTG}(3,n,3)$ is
2-spanning cyclable for all $n\geq 8$.

Now we consider extending the results to all odd $m>3$.  If we use the flat edges between the last
two columns and form one 2-factor with upward vertical fills and a second 2-factor with downward
vertical fills in $F$, then between these two 2-factors we separate $u_{0,0}$ from all the vertices in the two
new columns except $u_{2,0}$ and $u_{3,0}$.  These two vertices are handled by using $\rho^{-1}$
on the downward vertical fill.  The vertices from the original three columns are covered by vertical fills
on the various 2-factors employed to separate them from $u_{0,0}$.  This easily carries over to all
odd $m$ as we iterate two columns at a time.  This completes the proof.      \end{proof}

\begin{thm}\label{oddgen} Let $\ell>3$ be odd.  The graph $\mathrm{HTG}(m,n,\ell)$ is 2-spanning
cyclable for all odd $m$ and all even $n$ satisfying $n>\frac{3\ell^2-4\ell-3}{2}$.
\end{thm}
\begin{proof} Corollary \ref{asymp} corresponds to $m=1$.  To handle $m=3$ and larger values of odd
$m$, we need to carefully consider vertical fills.  In some sense, moving from $m=1$ to $m=3$ is the
crucial step.

Let $C_1$ and $C_2$ be the cycles in the proof of Theorem \ref{alpha} and recall that
$C_1=[v_{\ell-3},v_{\ell-2},v_{\ell-1},v_{n-1},v_{n-2},v_{n-3},v_{\ell-3}]$.  For the vertical fills for $m=3$, 
$C_1$ produces  the jump edges $[u_{2,n-3},u_{0,\ell-3}]$ and $[u_{2,n-1},u_{0,\ell-1}]$, and the flat 
edges $[u_{0,n-3},u_{1,n-3}]$ and $[u_{0,n-1},u_{1,n-1}]$.

\begin{picture}(200,300)(-50,-50)
\multiput(100,0)(20,0){3}{\multiput(0,0)(0,20){12}{\circle*{5}}}
\put(100,160){\line(0,1){20}}
\multiput(100,80)(0,40){3}{\line(2,-3){40}}
\put(100,0){\line(0,1){60}}
\qbezier(100,0)(120,90)(140,180)
\put(100,180){\line(0,1){40}}
\put(100,60){\line(1,0){20}}
\put(120,0){\line(0,1){60}}
\put(140,20){\line(0,1){40}}
\multiput(120,80)(0,40){2}{\line(1,0){20}}
\put(140,80){\line(0,1){20}}
\multiput(100,100)(0,40){2}{\line(1,0){20}}
\put(100,220){\line(1,0){20}}
\put(140,120){\line(0,1){40}}
\put(120,160){\line(1,0){20}}
\put(120,200){\line(0,1){20}}
\put(120,120){\line(0,1){20}}
\put(120,0){\line(1,0){20}}
\put(140,180){\line(0,1){40}}
\multiput(100,80)(0,40){2}{\line(0,1){20}}
\put(120,80){\line(0,1){20}}
\put(120,160){\line(0,1){40}}
\qbezier(140,0)(160,110)(140,220)
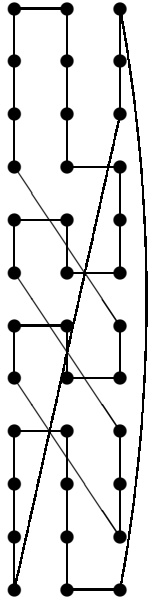
\put(40,-35){{\sc Figure 2.} A 2-factor in $\mathrm{HTG}(3,12,3)$}
\end{picture}

We want to know what happens to $C_1$ after vertical fills.  If we do a downward fill, 
$[u_{1,n-1},u_{1,n-2},u_{2,n-2},u_{2,n-1}]$ is the path from $u_{1,n-1}$ to $u_{2,n-1}$.  As we move
down from $j=n-3$, the first flat edge from $C_2$ is $[u_{0,n-\ell},u_{1,n-\ell}]$ which implies the
path from $u_{1,n-3}$ to $u_{2,n-3}$ starts down at $u_{1,n-3}$ until reaching $u_{1,n-\ell+1}$
followed by the edge to $u_{2,n-\ell+1}$ and then back up to $u_{2,n-3}$.  Let $C_3$ denote the
cycle obtained from $C_1$ by a downward fill and $F_3$ denote the resulting 2-factor. 

When doing an upward fill, the first flat edge from $C_2$ we encounter as we move up from $n-1$ is 
$[u_{0,\ell-4},u_{1,\ell-4}]$.  Thus, we know the replacement paths for an upward fill.  Let $C_4$ denote 
the cycle obtained from $C_1$ via an upward fill and $F_4$ the resulting 2-factor.

We use three automorphisms from the group $G$ acting on $\mathrm{HTG}(3,n,\ell)$.  Let $g_1\in G$
map $u_{0,\ell-3}$ to $u_{0,0}$, $g_2$ map $u_{0,\ell-1}$ to $u_{0,0}$ and $g_3$ map $u_{0,n-2}$
to $u_{0,0}$.

We first show that we may separate $u_{0,0}$ from all the other vertices from column 0.  The 2-factor
$g_1(F_3)$ separates $u_{0,0}$ from all vertices $u_{0,j}$ except for $j\in\{1,2,n-\ell,n-\ell+1,
n-\ell+2\}$.  The 2-factor $g_2(F_3)$ separates $u_{0,0}$ from $u_{0,1},u_{0,2},u_{0,n-\ell+1}
\mbox{ and }u_{0,n-\ell+2}$.  Finally, the only way that $g_3(F_3)$ does not separate $u_{0,0}$ and
$u_{0,n-\ell}$ is if $n=2\ell$, but this is not the case by hypothesis. 

We have found 2-factors that separate $u_{0,0}$ from all other vertices in column 0 in $\mathrm{HTG}(3,n,\ell)$.
We move to columns 1 and 2.  The 2-factor $g_1(F_4)$ separates $u_{0,0}$ from $u_{i,j}$ for $i=1,2$ and
$0\leq j\leq n-\ell-1$.  The 2-factor $g_3(F_4)$ separates $u_{0,0}$ and $u_{i,j}$ for $i=1,2$ and $j$ in
the range $\ell-1\leq j\leq n-2$.  Because $\ell-1<n-2\ell+3$, the two intervals overlap and the only vertices
not separated from $u_{0,0}$ are those for which $j=n-2\mbox{ and }n-1$.  Fortunately, $g_2(F_3)$
separates $u_{0,0}$ from those four missing vertices.

The preceding shows that $\mathrm{HTG}(3,n,\ell)$ is 2-spanning cyclable for $n$ satisfying the hypotheses.
We obtain the desired separability for all remaining odd values of $m$ by iteratively performing vertical
fills between the last two columns.  Note that the fills alternate between upward and downward as we
carry out the process.     \end{proof}

\section{Even Number Of Columns}

Again we examine cases for which there are exact results and cases for which there are asymptotic
results.  The first exact case is for $\ell=0$.

\begin{thm}\label{L0} The graph $\mathrm{HTG}(m,n,0)$ is 2-spanning cyclable if and only if
$n>4$ or $(m,n,0)=(2,4,0)$.
\end{thm}
\begin{proof} The graph $\mathrm{HTG}(2,4,0)$ is 2-spanning cyclable by Lemma \ref{zero}.
Next consider $\mathrm{HTG}(m,4,0)$ for even $m>2$.  We claim there is no 2-factor separating
$u_{0,0}$ and $u_{m-1,0}$.  If there was such a 2-factor, then the cycle $C_1$ containing $u_{0,0}$
would contain the 2-path $[u_{0,1},u_{0,0},u_{0,3}]$ and the cycle $C_2$ containing $u_{m-1,0}$
would contain the 2-path $[u_{m-1,1},u_{m-1,0},u_{m-1,3}]$.  The jump edge $[u_{0,1},u_{m-1,1}]$
cannot belong to the 2-factor because it would have one end in $C_1$ and the other end in $C_2$.
Thus, $C_1$ is forced to be the 4-cycle which is column 0 and $C_2$ is forced to be the 4-cycle which
is column $m-1$.  We conclude that the graph is not 2-spanning cyclable.

Next consider $\mathrm{HTG}(4,n,0)$ with even $n\geq 6$.  Let $C_1$ be the 6-cycle given by
$C_1=[u_{0,0},u_{0,1},u_{1,1},u_{1,0},u_{1,n-1},u_{0,n-1},u_{0,0}]$.  There is a cycle through
the remaining vertices given by 
\begin{eqnarray*}C_2&=&[u_{0,2},u_{3,2},u_{3,1},u_{3,0},u_{3,n-1},u_{2,n-1},u_{2,0},
u_{2,1},u_{2,2},u_{1,2},\ldots,u_{1,n-2},\\& & u_{2,n-2},\ldots,u_{2,3},u_{3,3},\ldots,u_{3,n-2},
u_{0,n-2},\ldots,u_{0,2}]. \end{eqnarray*}  The 2-factor $C_1\cup C_2$ separates $u_{0,0}$ from
every vertex not belonging to $C_1$.  It is straightforward to check that using the action of the group
$G$, we can find 2-factors separating $u_{0,0}$ from the five vertices of $C_1$ distinct from $u_{0,0}$.

All the vertices of the last two columns belong to $C_2$.  So if we add successive pairs of columns by
using vertical fills between the last two columns, all the new vertices do not belong to $C_1$ so
that we see easily that $\mathrm{HTG}(m,n,0)$ is 2-spanning cyclable for all even $m\geq 4$.  The
result now follows.    \end{proof}

\begin{thm}\label{L2} The graph $\mathrm{HTG}(m,n,2)$ is 2-spanning cyclable if and only if
$(m,n,2)=(2,4,2)$ or $m\geq 4$ is even and $n\geq 6$ is even.
\end{thm}
\begin{proof}  By interchanging $u_{m-1,0}$ and $u_{m-1,2}$, it is clear that $\mathrm{HTG}(m,4,2)$
and $\mathrm{HTG}(m,4,0)$ are isomorphic.  Theorem \ref{L0} then takes care of the claims for
$n=4$.  Lemma \ref{two} eliminates the remaining possibilities for $m=2$.  Hence, we consider
$m\geq 4$ and $n\geq 6$.  

Let $$C_1=[u_{0,0},u_{0,1},u_{1,1},\ldots u_{1,n-2},u_{2,n-2},
\ldots,u_{2,1},u_{3,1},\ldots,u_{3,n-2},u_{0,0}].$$  We leave it to the reader to easily verify that
there is a cycle $C_2$ through the complementary set of vertices.  Note that all the vertices on level $n-1$, 
all the vertices on level 0 distinct from $u_{0,0}$ and all the vertices of column 0, other than $u_{0,1}$,
belong to $C_2$ and are separated from $u_{0,0}$.  

Letting $\rho^{-1}$ act on $C_1\cup C_2$ yields a 2-factor in which $u_{0,0}$ is separated from all
vertices of the form $u_{i,j}$ for $i=1,2,3$ and $j\in\{1,2,\ldots,n-4\}$.  Using the 2-factor arising from
$\rho^{-2}$, we take care of the remaining vertices on levels $n-2$ and $n-3$.  Finally, the element of
$G$ mapping $u_{2,0}$ to $u_{0,0}$ gives a 2-factor separating $u_{0,0}$ and $u_{0,1}$.  It
follows that $\mathrm{HTG}(4,n,2)$ is spanning 2-cyclable.

If we start with an upward vertical fill between the last two columns, essentially the same verification
shows that the resulting $\mathrm{HTG}(6,n,2)$ is spanning 2-cyclable.  We continue with alternating
upward and downward vertical fills to complete the proof.    \end{proof}

\begin{thm}\label{evengen} Let $\ell\geq 4$ be even.  If $m$ is even and $n>\frac{\ell^2+2\ell-4}{2}$,
then $\mathrm{HTG}(m,n,\ell)$ is 2-spanning cyclable.
\end{thm}
\begin{proof} The case $m=2$ is covered by Corollary \ref{betamax}.  Consider $m=4$.  Let $C_1$ and
$C_2$ be the cycles forming the 2-factor in the proof of Theorem \ref{beta} and recall that
$$C_1=[u_{0,0},u_{0,n-1},u_{1,n-1},u_{1,0},u_{1,1},\ldots,u_{1,\ell-3},u_{0,\ell-3},u_{0,\ell-4},\ldots,u_{0,0}].$$ 
We use $C_1$ as the basis for extending the desired property to all even $m>2$

It is straightforward to see that if we form a 2-factor for $m=4$ using upward vertical fills and another 2-factor
for $m=4$ using downward vertical fills on the flat edges between the two columns, between the two 2-factors 
we separate $u_{0,0}$ from all vertices of the form $u_{i,j}$ for $0\leq i\leq 3$ and $j\in\{\ell-2,\ell-1,\ldots,
n-2\}$.  We may continue using the appropriate vertical fills to separate $u_{0,0}$ from all vertices on levels
$\ell-2$ through $n-2$ for all even $m$.

By taking the element of $G$ mapping $u_{0,\ell-4}$, we separate $u_{0,0}$ from all the vertices on levels
2 through $\ell-3$.  Similarly, using $u_{0,\ell-4}$ and $u_{0,n-2}$ in the role of $u_{0,0}$, we separate
$u_{0,0}$ from all vertices on levels 1 and $n-1$.   We require another argument for the vertices on level 0.

In $C_2$ there is a jump edge from $u_{1,n-2-\ell}$ to $u_{0,n-2}$ and the subpath from $u_{0,n-2}$ is
$[u_{0,n-2},u_{0,n-3},u_{0,n-4}]$.  Thus, if we do alternating downward and upward fills starting with downward,
then $u_{0,n-2}$ belongs to the modified $C_2$ and $u_{i,n-2}$ belongs to the modified $C_1$ for all
$1\leq i\leq m-2$.  Hence, if we use the automorphism in $G$ mapping $u_{0,n-2}$ to $u_{0,0}$, we have
a  2-factor separating $u_{0,0}$ and $u_{i,0}$ for $1\leq i\leq m-2$ for all $\mathrm{HTG}(m,n,\ell)$ in the
desired parameter range.  The element of $G$ that maps $u_{1,1}$ to $u_{0,0}$ in $\mathrm{HTG}(2,n,\ell)$
maps $u_{0,\ell+1}$ to $u_{1,0}$ which implies $u_{0,0}$ and $u_{1,0}$ are separated in the latter 2-factor.
The separation of $u_{0,0}$ and $u_{m-1,0}$ continues as we perform successive vertical fills.
This completes the proof.    \end{proof}

\section{Conclusion}

We have seen a hint that the problem of determining precisely which honeycomb toroidal graphs are
2-spanning cyclable is messy.  We sense this because there appear to be a few sporadic values of $n$ and $m$
for which $\mathrm{HTG}(m,n,\ell)$ is not 2-spanning cyclable given a fixed $\ell$.  The value $\ell=3$
illustrates this point well.

Consequently, we determined precise solutions for some small values of $\ell$ and then obtained asymptotic
results for the remaining values of $\ell$.  The asymptotic results are quadratic in $\ell$ and as we saw
they are not tight.


\begin{thebibliography}{9999}
\bibitem{A1} B. Alspach, Honeycomb Toroidal Graphs, {\sl Bull. Inst. Combin. Math.} {\bf 91} 
(2021), 94--114.
\bibitem{A2} B. Alspach, C. C. Chen and M. Dean, Hamilton paths in Cayley graphs on generalized 
dihedral groups, {\sl Ars Math. Contemp.} {\bf 3} (2010), 29--47.
\bibitem{L1} C.-K. Lin, J.-M. Tan, L.-H. Hsu and T.-L. Kung, Disjoint cycles in hypercubes with prescribed
vertices in each cycle, {\sl Discrete Appl. Math.} {\bf 161} (2013), 2992--3004.
\bibitem{M1} G.M. Megson, X. Yang and X. Liu, Honeycomb tori are Hamiltonian, {\sl Inform. Process. Lett.}
{\bf 72} (1999), 99--103. 
\bibitem{Q1} H. Qiao, E. Sabir and J. Meng, The spanning cyclability of Cayley graphs generated
by transposition trees, {\sl Discrete Appl. Math.} {\bf 328} (2023), 60--69.
\bibitem{W1} E. Weisstein, Frobenius Number, from {\sl MathWorld}. A Wolfram Web Resource.
https://mathworld.wolfram.com/FrobeniusNumber.html.
\bibitem{Y1} M.-C. Yang, L.-H. Hsu, C.-N. Hung and E. Cheng, 2-Spanning cyclability problems of some
generalized Petersen graphs,  {\sl Discuss. Math. Graph Theory} {\bf 40} (2020), 713--731.

\end{thebibliography}
\end{document}